\documentclass{amsart}
\usepackage{amssymb}
\usepackage[utf8]{inputenc}
\usepackage{hyperref}
\usepackage{bm}
\usepackage{comment}
\usepackage{xcolor}

\newcommand{\Z}{{\ensuremath{\mathbb{Z}}}}

\theoremstyle{theorem}
\newtheorem{Theorem}{Theorem}
\newtheorem{Corollary}[Theorem]{Corollary}
\newtheorem{Lemma}[Theorem]{Lemma}
\newtheorem{Remark}[Theorem]{Remark}
\newtheorem{Proposition}[Theorem]{Proposition}
\theoremstyle{definition}

\newtheorem{Definition}[Theorem]{Definition}
\newtheorem{Example}[Theorem]{Example}

\newtheorem{Question}[Theorem]{Question}

\makeatletter
\@namedef{subjclassname@2020}{%
  \textup{2020} Mathematics Subject Classification}
\makeatother

\title{On sums of $gr$-PI algebras}
\author[P. Fagundes]{Pedro Fagundes}
\address{Department of Mathematics, State University of Campinas, 651 S\'ergio Buarque de Holanda, Cidade Universit\'aria ``Zeferino Vaz'', Bar\~ao Geraldo, 13083-859 Campinas, SP, Brazil, 
		}\email{pedro.fagundes@ime.unicamp.br}
\author[P. Koshlukov]{Plamen Koshlukov}
\address{Department of Mathematics, State University of Campinas, 651 S\'ergio Buarque de Holanda, Cidade Universit\'aria ``Zeferino Vaz'', Bar\~ao Geraldo, 13083-859 Campinas, SP, Brazil,}\email{plamen@unicamp.br}	
\thanks{P.\ Fagundes was supported by Grant 2019/16994-1 and Grant 2022/05256-2, São Paulo Research Foundation (FAPESP). P Koshlukov was partially supported by FAPESP grant No.~2018/23690-6 and by CNPq grant No.~302238/2019-0.}

\subjclass[2020]{16R50, 16W50, 16R99}

\keywords{
PI algebras, Sums of graded algebras, Polynomial identities, Graded codimensions}

\begin{document}

\maketitle

\begin{abstract}
    Let $A=B+C$ be an associative algebra graded by a group $G$, which is a sum of two homogeneous subalgebras $B$ and $C$. We prove that if $B$ is an ideal of $A$, and both $B$ and $C$ satisfy graded polynomials identities, then the same happens for the algebra $A$. We also introduce the notion of graded semi-identity for the algebra $A$ graded by a finite group and we give sufficient conditions on such semi-identities in order to obtain the existence of graded identities on $A$. We also provide an example where both subalgebras $B$ and $C$ satisfy graded identities while $A=B+C$ does not. Thus the theorem proved by K\c{e}pczyk in 2016 does not transfer to the case of group graded associative algebras. A variation of our example shows that a similar statement holds in the case of group graded Lie algebras. We note that there is no known analogue of K\c{e}pczyk's theorem for Lie algebras.
\end{abstract}

\section{\bf Introduction}

If $A$ and $B$ are two associative PI rings (or associative PI algebras over the same field $F$), how can we construct new PI rings ($F$-algebras) starting from the given ones? This question is natural and rather important. On the other hand in such a general setting one can hardly expect a satisfactory answer to it. Below we will speak of algebras but virtually everything will be valid for rings as well. The well known theorem due to Birkhoff states that a class of algebras is a variety (that is, it is defined by polynomial identities) if and only if it is closed under taking subalgebras, homomorphic images, and direct products. The theorem of Birkhoff holds for arbitrary algebraic systems but it is nevertheless a rather general construction, albeit of little practical use. The first nontrivial example of such a construction was given by A. Regev in 1972 in \cite{regev}, he proved that the tensor product of two PI algebras over a field is once again a PI algebra. This theorem has very many important consequences. One of them was obtained by Regev, in the same paper, and it states that the codimension sequence of a PI algebra cannot grow very fast: its growth is limited by an exponential function. 

Another important construction is the sum: if $B$ and $C$ are subrings of a ring $A$ such that $A=B+C$, can one claim that $A$ is PI assuming that $B$ and $C$ are PI? This question is attributed to Beidar and Mikhalev \cite{BMi}, though it first appeared, implicitly, in a paper by Kegel, in the early sixties \cite{Keg}. It was proved in that paper that if $B$ and $C$ are nilpotent then $A$ is also nilpotent. In case both $B$ and $C$ are commutative rings then $A$ need not be commutative: Bahturin and Giambruno \cite{BGi} proved that in this case $A$ is PI and it satisfies the identity $[x,y][z,t]=0$. Here and in what follows $[a,b]=ab-ba$ is the usual commutator (bracket). Felzenszwalb, Giambruno, and Leal in \cite{FGL} studied various conditions on $B$ and $C$ that ensure that $A$ is PI. In \cite{KPu} K\c{e}pczyk and Puczylowski proved that if $B$ is PI and $C$ is nil of bounded index then $A$ is PI. Finally in 2016, K\c{e}pczyk in \cite{Kep} gave the positive answer to the question whether $A$ is PI provided that $B$ and $C$ are PI. 

In this paper we study the group graded version of the above question. Group gradings on algebras represent a major tool in PI theory. Their usage in the study of PI algebras started with the celebrated works of Kemer, see for an account Kemer's monograph \cite{kemer}. There has been significant interest in group gradings on algebras and the corresponding graded identities since then. The interested reader is directed to consult the monograph \cite{eldkoc} and its references for future reading concerning gradings, and graded identities. 

We study the following question.
\begin{Question}\label{question}
Let $A$ be an algebra graded by a group $G$. Assume $B$ and $C$ are homogeneous (in the grading) subalgebras of $A$ such that $A=B+C$. If $B$ and $C$ both satisfy graded polynomial identities, does $A$ satisfy one? 
\end{Question}

We give an example in this paper that shows the above question, as stated, has a negative answer. Therefore one is led to consider a modification of it: What conditions must we impose on $A$, $B$, $C$ and/or the group grading in order to obtain the analogue of K\c{e}pczyk's theorem \cite{Kep} in the group graded case? It seems likely to us that this will require developing further the structure theory of graded rings, which is not the goal of the present paper.

The paper is organized in the following way. We start, in the next section, with recalling some notions needed later on. We consider algebras graded by monoids, and we show, by using the main  theorem from \cite{Kep} that if $A=B+C$ is graded by a monoid $M$, $B$ and $C$ are homogeneous subalgebras, and $B$ and $C$ both satisfy graded identities in neutral variables only then $A$ satisfies a graded identity in neutral variables. We also construct an example of a $\mathbb{Z}_2$-graded algebra $A$ which is a sum of two homogeneous subalgebras $B$ and $C$ where both $B$ and $C$ satisfy graded identities but $A$ does not. Furthermore we show that one can modify slightly the example in order to get $A=B\oplus C$ while the same conclusion holds. In addition to that we provide a similar example for graded Lie algebras. In section 3 we deal with algebras graded by a group $G$. We consider an arbitrary class of $G$-graded algebras that is closed under graded homomorphisms and direct powers. Following ideas from \cite{KPu}, we prove that if every algebra in the class has a nontrivial homogeneous ideal that satisfies a graded identity then every algebra in the class is graded PI. As a consequence we obtain a partial answer to Question \ref{question}: If  $B$ is a homogeneous ideal of $A$ and if both $B$ and $C$ are graded PI then so is $A$. In section 4 we study the behaviour of the graded semi-identities, a notion introduced in the non-graded case in \cite{FGL}. We employ ideas that go back to the celebrated paper by Regev, \cite{regev}, and compute upper bounds for the graded codimensions. This allows us to prove that if $A$ satisfies a graded semi-identity of certain type, and the neutral (in the grading) component of $C$ is PI then $A$ is graded PI. Furthermore we find an upper bound for the degree of a graded identity satisfied by $A$. Finally we comment on the variation of Question \ref{question} for algebras with involution.

\section{\bf Preliminaries}

Let $M$ be a monoid (with the multiplicative notation) and let $V$ be a vector space over a field $F$. A {\it grading} on $V$ by the monoid $M$ is a decomposition of $V$ into a direct sum of vector subspaces $V=\bigoplus_{m\in M} V_{m}$. The subspaces $V_{m}$ are called the {\it homogeneous components} of $V$. An element $v\in V_{m}$ is called a {\it homogeneous element} of degree $m$. A subspace $U$ of $V$ is called {\it homogenous} if one can write $U=\bigoplus_{m\in M}(U\cap V_{m})$.

We will denote by $A$ an associative algebra over $F$. A  grading on the algebra $A$ by a monoid $M$ is defined as in the vector space case above, requiring additionally that $A_{m_{1}}A_{m_{2}}\subset A_{m_{1}m_{2}}$ for all $m_{1}$, $m_{2}\in M$, where $A_{m_{i}}$ is the corresponding homogeneous component of $A$. We define homogeneous subalgebras and homogeneous (left, right, two sided) ideals of $A$ as subalgebras and (left, right, two sided) ideals of $A$ which are homogeneous as subspaces, respectively. By an ideal of $A$ we always mean a two-sided ideal. 

For each $m\in M$ we take $X^{(m)}=\{x_{1}^{(m)},x_{2}^{(m)},\dots\}$ as a set of noncommutative variables. Putting $X=\bigcup_{m\in M}X^{(m)}$, we consider the free associative algebra $F\langle X \rangle$, which has a basis as a vector space given by all noncommutative words in the alphabet $X$. The multiplication on $F\langle X \rangle$ is given by juxtaposition. We can consider an $M$-grading on $F\langle X\rangle$ by defining the homogeneous degree $\deg(x_{i}^{(m)})=m$ for all $x_{i}^{(m)}\in X^{(m)}$ and extending it  to all monomials $\mathbf{m}$ as the product of the homogeneous degrees of the variables that appear in $\mathbf{m}$, in the order that they are occurring. Denoting by $F\langle X \rangle^{(m)}$ the subspace of $F\langle X \rangle$ generated by all monomials of homogeneous degree $m$, we have that $F\langle X \rangle=\bigoplus_{m\in M}F\langle X \rangle^{(m)}$ defines an $M$-grading on $F\langle X \rangle$. The free associative algebra with such $M$-grading will be denoted by $F\langle X|M \rangle$, it is called the {\it free $M$-graded associative algebra}. 

Given a graded algebra $A$, we say that a nonzero polynomial $f\in F\langle X|M \rangle$ is an {\it $M$-graded polynomial identity} for the algebra $A$ if 
\[
    f\in \bigcap \ker(\phi) ,
\]
where the intersection runs over all graded homomorphisms $\phi\colon F\langle X|M \rangle\rightarrow A$. By a graded homomorphism $\phi\colon A\rightarrow B$ between two $M$-graded algebras we mean an algebra homomorphism satisfying $\phi(A_{m})\subset B_{m}$ for all $m\in M$. 

An algebra satisfying some graded polynomial identity will be called a {\it $gr$-PI algebra}. We denote by $Id_{M}(A)$ the set of all $M$-graded polynomial identities of $A$. Note that $Id_{M}(A)$ is an ideal which is invariant under graded endomorphisms, and that every such ideal coincides with the ideal of graded identities for some graded algebra. 

The above definitions and constructions are analogously valid for gradings by a group. 

One may try to reduce the answer of Question \ref{question} to the ordinary case through the K\c{e}pczyk's Theorem \cite[Theorem 5]{Kep}. This can be done by looking for sufficient conditions on the graded identities on $B$ and $C$ which imply the existence of ordinary identities on these subalgebras. As an example of such sufficient condition we recall the following result from \cite{BGR,BCo}.

\begin{Theorem}
Let $G$ be a finite group and let $A$ be a $G$-graded algebra. If the neutral component $A_{1}$ is a PI-algebra, then $A$ is a PI-algebra.
\end{Theorem}

As a direct consequence we have the following corollary.

\begin{Corollary}
Let $G$ be a finite group and let $A=B+C$ be an algebra that is the sum of two homogeneous subalgebras. If $B$ and $C$ satisfy graded polynomial identities in neutral variables, then $A$ is a PI-algebra.
\end{Corollary}

We give below a slight improvement of the corollary above.

\begin{Lemma}\label{l}
Let $M$ be a monoid and let $V$ be an $M$-graded vector space. Let $V=V_{1}+V_{2}$ be a sum of two homogeneous subspaces. Then $V_{m}=V_{1,m}+V_{2,m}$ for each $m\in M$, where $V_{i,m}$ is the homogeneous component of degree $m$ of $V_{i}$, $i=1$, 2.
\end{Lemma}

\begin{proof}
Let $m\in M$. Since $V_{1,m}=V_{m}\cap V_{1}$ and $V_{2,m}=V_{m}\cap V_{2}$ it follows that $V_{1,m}+V_{2,m}\subset V_{m}$. Reciprocally let $v\in V_{m}$. Hence $v=v_{1}+v_{2}$ for some $v_{1}\in V_{1}$ and $v_{2}\in V_{2}$. Writing 
\[
v_{1}=v_{1,m_{1}}+\cdots+v_{1,m}+\cdots +v_{1,m_{n}} \ \mbox{and} \ v_{2}=v_{2,m_{1}}+\cdots+v_{2,m}+\cdots+ v_{2,m_{n}}
\]
such that $v_{1,m_{i}}\in V_{1,m_{i}}$, $v_{1,m}\in V_{1,m}$, $v_{2,m_{i}}\in V_{2,m_{i}}$, $v_{2,m}\in V_{2,m}$, then 
\[
0=v-v_1-v_2= (v-v_{1,m}-v_{2,m})-(v_{1,m_{1}}+v_{2,m_{1}})-\cdots - (v_{1,m_{n}}+v_{2,m_{n}}).
\]
This implies $v=v_{1,m}+v_{2,m}$. Therefore $V_{m}=V_{1,m}+V_{2,m}$.
\end{proof}

Let $A$ be an algebra graded by a monoid $M$. Assume that $A=B+C$ is a sum of two homogeneous subalgebras. Then Lemma \ref{l} gives us that $A_{1}=B_{1}+C_{1}$. Since $B_{1}$ and $C_{1}$ are subalgebras of $A_{1}$, the K\c{e}pczyk's Theorem leads us to the following result.  

\begin{Theorem}\label{t1}
Let $A=B+C$ be an algebra graded by a monoid $M$ such that $B$ and $C$ are two homogeneous subalgebras. If $B$ and $C$ satisfy polynomial identities in neutral variables, then the same holds for the algebra $A$.
\end{Theorem}

We finish this section with the construction of an example that shows that in general one can not expect an affirmative answer to Question \ref{question}.
We denote by $D$ the free associative algebra $F\langle X\rangle$ where we assume $X$ an infinite countable set. Define $A=M_2(D)$ the $2\times 2$ matrix algebra with entries from $D$, and set 
\[
A_0 = \begin{pmatrix} D&0\\ 0&D\end{pmatrix}, \qquad A_1=\begin{pmatrix}0&D\\ D&0\end{pmatrix}.
\]
Then clearly $A=A_0\oplus A_1$ is a $\mathbb{Z}_2$-grading on $A$. We define now two subalgebras $B$ and $C$ of $A$:
\[
B= \begin{pmatrix} D&D\\ 0&D\end{pmatrix}, \qquad C=\begin{pmatrix}D&0\\ D&D\end{pmatrix}.
\]
\begin{Lemma}
The subalgebras $B$ and $C$ of $A$ are homogeneous in the grading, $A=B+C$, and both $B$ and $C$ satisfy the graded identity $x_1^{(1)}x_2^{(1)}=0$.
\end{Lemma}
\begin{proof}
All statements are immediate.
\end{proof}

Now we prove that $A$ satisfies no $\mathbb{Z}_2$-graded identities. Observe that $A$, $B$, $C$ are not PI algebras since they contain a free associative algebra as a subalgebra. 

\begin{Proposition}\label{no identity}
The $\mathbb{Z}_2$-graded algebra $A$ satisfies no graded identities.
\end{Proposition}

\begin{proof}
Suppose, on the contrary, that $A$ does satisfy a graded identity $f$. Without loss of generality we can assume $f$ multilinear, $f=f(x_1^{(0)},\ldots, x_m^{(0)}, x_1^{(1)},\ldots,x_k^{(1)})$. We shall use a \textsl{generic} argument. We take the elements
\[
a_i=\begin{pmatrix} u_i&0\\ 0&v_i\end{pmatrix}, \qquad b_i=\begin{pmatrix}0&w_i\\ t_i&0\end{pmatrix}, \qquad i\ge 1
\]
in $A$ where $u_i$, $v_i$, $w_i$, $t_i$ are distinct variables in the set $X$. These generic elements multiply as follows:
\begin{align*}
a_i a_j=\begin{pmatrix} u_i u_j&0\\ 0&v_i v_j\end{pmatrix},\qquad &a_ib_j=\begin{pmatrix} 0&u_i w_j\\ v_i t_j&0\end{pmatrix},\\ b_ib_j=\begin{pmatrix} w_i t_j&0\\ 0&t_i w_j\end{pmatrix},\qquad &b_ja_i=\begin{pmatrix} 0&w_j t_i\\ t_j u_i&0\end{pmatrix}.
\end{align*}
Take a monomial $\bm{m}=X_1^{(0)}x_{j_1}^{(1)}X_2^{(0)}x_{j_2}^{(1)}\cdots x_{j_k}^{(1)} X_{k+1}^{(0)}$ where $(j_1,\ldots,j_k)$ is a permutation of $(1,\ldots, k)$, and the $X_i^{(0)}$ are monomials that do not contain variables from $X^{(1)}$ (some of the $X_i^{(0)}$ may be empty). Suppose that $\bm{m}$ participates in $f$ with non-zero coefficient. As $f$ is assumed a graded identity for $A$ then $f(a_1,\ldots, a_m, b_1,\ldots, b_k)=0$ in $A$. 

Let us evaluate $\bm{m}$ on $a_1$, \dots, $a_m$, and $b_1$, \dots, $b_k$. Depending on the parity of $k$ we get a matrix in either $A_0$ or $A_1$. Assume $k$ even. Then at position $(1,1)$ of the resulting matrix we will have an entry of the type
\[
\alpha_1 w_{j_1} \beta_2 t_{j_2} \alpha_3 w_{j_3} \beta_4 t_{j_4} \cdots 
\]
where $\alpha_1$ is the product of the entries $u_i$ at position $(1,1)$ of the matrices $a_i$ that appear in the substitution for $X_1^{(0)}$, in their respective order. Similarly $\beta_2$ is the product of the entries $v_i$ that come from the matrices in the second block, $X_2^{(0)}$, and so on, alternating the $(1,1)$ and $(2,2)$ entries of these blocks consecutively. 

Since we obtain a monomial in the free associative algebra $D$, it must cancel out with the monomial coming from some other term of $f$. But our monomial comes from only one term of $f$, namely from $\bm{m}$. Thus the resulting monomial in $D$ cannot cancel out, and this proves that the coefficient of $\bm{m}$ must be 0. Hence if $k$ is even we are done. When $k$ is odd the argument is analogous, and we omit it.
\end{proof}

Thus we proved that Question \ref{question} has a negative answer in general. This justifies our interest in looking for additional conditions that ensure the positive answer to that question. 

\begin{Remark}\label{example for direct sum}
One could expect that Question \ref{question} has a positive answer if we require $A=B\oplus C$ instead of $A=B+C$. However, even in the direct sum case we have the following example: we consider the $\mathbb{Z}_{2}$-graded algebra $A=A_{0}\oplus A_{1}$ as in the previous example, and we take
\[
B=\begin{pmatrix}
D&D\\
0&0
\end{pmatrix} \ \mbox{and} \ C=\begin{pmatrix}
0&0\\
D&D
\end{pmatrix}.
\]
Clearly $A=B\oplus C$, and both $B$ and $C$ are homogeneous subalgebras satisfying $x_{1}^{(1)}x_{2}^{(1)}=0$.
\end{Remark}

\begin{Remark}
The same two examples given above transfer to the graded Lie case. Indeed, let $A=A_{0}\oplus A_{1}$ as before, and consider $A^{(-)}$, the Lie algebra structure given on $A$ by setting $[a,b]=ab-ba$, for all $a$, $b\in A$. Concerning the first example, we write $A^{(-)}=B^{(-)}+C^{(-)}$, and one can easily see that $B^{(-)}$ and $C^{(-)}$ are homogeneous Lie subalgebras of $A^{(-)}$, both satisfying $[x_{1}^{(1)},x_{2}^{(1)}]=0$. It remains to prove that $A^{(-)}$ does not satisfy $\mathbb{Z}_{2}$-graded Lie identities. Assume the contrary, and consider the embedding of the free $\mathbb{Z}_{2}$-graded Lie algebra into $F\langle X|\mathbb{Z}_{2}\rangle^{(-)}$. Hence, we would have a $\mathbb{Z}_{2}$-graded (associative) identity for $A$, which can not occur in light of Proposition \ref{no identity}.
The example from Remark \ref{example for direct sum} is treated analogously.
\end{Remark}

\section{\bf Sums of $gr$-PI ideals and $gr$-PI subalgebras}

In this section we study graded algebras which are a sum of a gr-PI ideal and a gr-PI subalgebra and we give a partial answer to Question \ref{question}. 

Throughout this section $F$ will denote an arbitrary field, $G$ an arbitrary group and $A$ an associative algebra over $F$ graded by $G$. We recall that the operations on some direct power $\prod A$ of $A$ are given by the operations on $A$ component-wise. We can extend the $G$-grading of $A$ to the algebra $\prod A$ as follows:
\[
\prod A=\bigoplus_{g\in G} \mathbf{A}_{g} \ \mbox{such that} \ \mathbf{A}_{g}=\prod A_{g}
\]
where $A_{g}$ is the homogeneous component of degree $g$ of $A$.

\begin{Lemma}\label{l11}
If $A$ is a $gr$-PI algebra, then $\prod A$ is also a $gr$-PI algebra.
\end{Lemma}

\begin{proof}
It is enough to note that if $A$ satisfies $f=0$, then $\prod A$ also does. 
\end{proof}

\begin{Lemma}\label{ll}
Suppose  that $A$ is not a gr-PI algebra. Then some direct power of $A$ contains a homogeneous subalgebra that is isomorphic to a free graded associative algebra.
\end{Lemma}

\begin{proof}
Take $H=F\langle X|G \rangle\setminus\{0\}$. By hypothesis, for each $h\in H$ we have $h\notin Id_{G}(A)$. Hence there exists some graded homomorphism $\phi_{h}\colon F\langle X|G \rangle \rightarrow A$ such that $\phi_{h}(h)\neq0$. We define $\psi\colon F\langle X|G \rangle\rightarrow \prod_{h\in H} A$ by $\psi(f)=(\phi_{h}(f))_{h}$. It is easy to check that $\psi$ is a graded homomorphism. Moreover, if $h\in H$ then $\phi_{h}(h)\neq0$ and then $\psi(h)\neq0$. We conclude that $\psi$ is an embedding of $F\langle X|G \rangle$ into $\prod_{h\in H} A$.
\end{proof}

\begin{Lemma}\label{l2}
Let $F\langle X|G \rangle$ be the free graded associative algebra and let $K$ be a homogeneous left ideal of  $F\langle X|G \rangle$. Then $K$ does not satisfy any graded polynomial identity.
\end{Lemma}

\begin{proof}
Assume that $K$ satisfies some graded polynomial identity $f(x_{1}^{(g_{1})},\dots,x_{n}^{(g_{n})})$. Without loss of generality we can  assume that $f$ is multilinear of degree $n$. Let $w\in K\setminus\{0\}$  be a homogeneous element such that $\deg(w)=g$. For  $h_{i}=g_{i}g^{-1}$, $i=1$, \dots, $n$, we consider the variables $x_{i}^{(h_{i})}$. Hence  $\deg(x_{i}^{(h_{i})}w)=g_{i}$ and then  $f(x_{1}^{(h_{1})}w,\dots, x_{n}^{(h_{n})}w)=0$.

On the other hand, we write $f$ and $w$ as follows:
\[
f=m_{1}+\cdots + m_{n}, \qquad w=w_{1}+\cdots + w_{l}
\]
as sums of their distinct monomials, respectively. Then 
\begin{equation}\label{ee}
f(x_{1}^{(h_{1})}w,\dots, x_{n}^{(h_{n})}w)=\sum_{i=1}^{n}\sum_{j_{1},\dots,j_{n}=1}^{l} m_{i}(x_{1}^{(h_{1})}w_{j_{1}},\dots,x_{n}^{(h_{n})}w_{j_{n}}).
\end{equation}
We claim that $f(x_{1}^{(h_{1})}w,\dots, x_{n}^{(h_{n})}w)$ is a nonzero polynomial and this will lead us to a contradiction. To this end, it is enough to prove that the monomials in (\ref{ee}) are pairwise distinct.  

For a fixed index $i$, we must have 
\[
\bm{m_{p}}=m_{i}(x_{1}^{(h_{1})}w_{p_{1}},\dots,x_{n}^{(h_{n})}w_{p_{n}})\neq m_{i}(x_{1}^{(h_{1})}w_{q_{1}},\dots,x_{n}^{(h_{n})}w_{q_{n}})=\bm{m_{q}}
\]
provided that the $n$-tuples $(p_1,\ldots, p_n)\ne (q_1,\ldots, q_n)$. Indeed, without loss of generality we may assume $p_{1}\neq q_{1}$ and $m_{1}=x_{1}^{(g_{1})}\cdots x_{n}^{(g_{n})}$. Then we write 
\[
\bm{m_{p}}=x_{1}^{(g_{1})}w_{p_{1}}w^{\prime} \ \mbox{and} \ \bm{m_{q}}=x_{1}^{(g_{1})}w_{q_{1}}w^{\prime\prime}
\]
and since $w_{p_{1}}\neq w_{q_{1}}$, then $\bm{m_{p}}\neq \bm{m_{q}}$.

It remains to analyse the monomials 
\[
\bm{m_{p}}= m_{i}(x_{1}^{(h_{1})}w_{p_{1}},\dots,x_{n}^{(h_{n})}w_{p_{n}}) \ \mbox{and} \ \bm{m_{q}}=m_{j}(x_{1}^{(h_{1})}w_{q_{1}},\dots,x_{n}^{(h_{n})}w_{q_{n}})
\]
for $i\neq j$. We start by rewriting the monomials $m_{i}(x_{1}^{(g_{1})},\dots,x_{n}^{(g_{n})})=m x_{i}^{(g_{i})}m^{\prime}$ and $m_{j}(x_{1}^{(g_{1})},\dots,x_{n}^{(g_{n})})=m x_{j}^{(g_{j})}m^{\prime\prime}$, where $m$ is the (possibly empty) common monomial at the beginning of the monomials $m_{i}$ and $m_{j}$. Therefore we suppose that the monomial $\bm{m_{p}}$ starts with $m(x_{1}^{(h_{1})}w_{p_{1}},\dots,x_{n}^{(h_{n})}w_{p_{n}})x_{i}^{(g_{i})}w_{p_{i}}$ and $\bm{m_{q}}$ starts with \linebreak  $m(x_{1}^{(h_{1})}w_{q_{1}},\dots,x_{n}^{(h_{n})}w_{q_{n}})x_{j}^{(g_{j})}w_{q_{j}}$. Since $x_{i}^{(g_{i})}\neq x_{j}^{(g_{j})}$, then we must have $\bm{m_{p}}\neq \bm{m_{q}}$, even if $m(x_{1}^{(h_{1})}w_{p_{1}},\dots,x_{n}^{(h_{n})}w_{p_{n}})=m(x_{1}^{(h_{1})}w_{q_{1}},\dots,x_{n}^{(h_{n})}w_{q_{n}})$.
\end{proof}

Before stating the main theorem of this section, we recall that a graded algebra $A$ is called $gr$-prime if the product of two nonzero homogeneous ideals of $A$ is still nonzero.

\begin{Theorem}\label{t1}
Let $\mathcal{A}$ be a class of $G$-graded $F$-algebras closed under graded homomorphic images and direct powers. Assume that every gr-prime algebra in $\mathcal{A}$ has a nonzero homogeneous ideal satisfying some graded polynomial identity. Then $\mathcal{A}$ is a class of gr-PI algebras.
\end{Theorem}

\begin{proof}
Assume that some algebra $A\in \mathcal{A}$ is not gr-PI.  
Then Lemma \ref{ll} gives us the existence of some homogeneous subalgebra $K$ of $\prod A$ which is isomorphic to $F\langle X|G \rangle$. By the Zorn Lemma, there exists a homogeneous ideal $I$ of $\prod A$ that is maximal with respect to the property $K\cap I=\{0\}$. Since $\mathcal{A}$ is closed under  homomorphic images and direct powers then $\bar{A}=(\prod A)/I\in\mathcal{A}$. We claim that $\bar{A}$ is a gr-prime algebra. In fact, let $J_{1}/I$ and $J_{2}/I$ be nonzero homogeneous ideals of $\bar{A}$. Hence there exist $x\in K\cap J_{1}\setminus\{0\}$ and $y\in K\cap J_{2}\setminus \{0\}$, and then $0\neq (x+I)(y+I)\in (J_{1}/I)(J_{2}/I)$, which proves our claim. 
Therefore $\bar{A}$ has a nonzero homogeneous ideal $\bar{L}=L/I$ satisfying a graded polynomial identity. Since $\bar{L}\neq0$, the maximality of $I$ implies the existence of a nonzero homogeneous element $x\in L\cap K$, that is, $x\notin I$. Recalling that
\[
\bar{K}=(K+I)/I\cong K/(K\cap I)\cong K\cong F\langle X|G \rangle,
\]
we have that $\bar{K}\bar{x}\subset \bar{L}$. In other words, we have a nonzero homogeneous left ideal of $\bar{K}$ which satisfies a graded polynomial identity. However this is an absurd according to Lemma \ref{l2}.
\end{proof}

As a corollary we have the following partial answer to Question \ref{question}.

\begin{Corollary}
Let $A$ be a $G$-graded $F$-algebra such that $A=B+C$ where $B$ is a nonzero homogeneous ideal of $A$ and $C$ is a homogeneous subalgebra of $A$. Moreover assume that $B$ and $C$ both satisfy graded polynomial identities. Then $A$ is a gr-PI algebra.
\end{Corollary}

\begin{proof}
Consider the class $\mathcal{A}$ of all $G$-graded algebras $A=B+C$ 
where $B$ is a nonzero homogeneous $gr$-PI ideal and $C$ is a homogeneous $gr$-PI subalgebra (notice that we may have $C=\{0\}$). Hence one can see that $\mathcal{A}$ contains the class of all $G$-graded algebras satisfying some graded polynomial identity and actually we aim to show that these classes are the same. 

Given $A\in\mathcal{A}$, note that $\prod A=\prod B +\prod C$, $\prod B$ is a nonzero homogeneous ideal of $\prod A$ and $\prod C$ is homogeneous subalgebra of $\prod A$. Moreover, Lemma \ref{l11} gives us that both $\prod B$ and $\prod C$ satisfy graded polynomial identities provided $B$ and $C$ also do. Hence $\prod A \in \mathcal{A}$. 
Now let $\varphi\colon A \rightarrow A_{2}$ be a graded epimorphism where $A=B+C\in\mathcal{A}$. Then $A_{2}=\varphi(B)+\varphi(C)$ where $\varphi(B)$ is a homogeneous $gr$-PI ideal of $A$ and $\varphi(C)$ is a homogeneous $gr$-PI subalgebra of $A$. In case $\varphi(B)=\{0\}$, then $A_{2}$ is a $gr$-PI algebra and thus it belongs to $\mathcal{A}$. Otherwise $\varphi(B)$ is a nonzero ideal and we also have $A_{2}\in \mathcal{A}$. 

Since every algebra in $\mathcal{A}$ contains some nonzero homogeneous $gr$-PI ideal, then we can apply Theorem \ref{t1} to get that $\mathcal{A}$ is a class of $gr$-PI algebras. In particular, $A\in\mathcal{A}$ is a $gr$-PI algebra.
\end{proof}

We note that no additional information is given on the identity satisfied by $A$. However, this is not the case when we assume that  either $B$ or $C$ satisfies some ordinary polynomial identity.

\begin{Corollary}
Let $B$ be a homogeneous ideal satisfying some ordinary polynomial identity $f(x_{1},\dots,x_{m})$ and let $C$ be a homogeneous subalgebra satisfying some multilinear graded identity $g(x_{1}^{(g_{1})},\dots,x_{n}^{(g_{n})})$. Then the algebra $A$ satisfies the following graded identity 
\[
f(g(x_{11}^{(g_{1})},\dots,x_{1n}^{(g_{n})}),\dots,g(x_{m1}^{(g_{1})},\dots,x_{mn}^{(g_{n})})).
\]
\end{Corollary}

\begin{proof}
We evaluate each homogeneous variable of degree $h$ by some element $a_{h}\in A_{h}$. By Lemma \ref{l}, we can write $a_{h}=b_{h}+c_{h}$ where $b_{h}\in B_{h}$ and $c_{h}\in C_{h}$. Since the polynomial $g$ is multilinear, it turns out that each evaluation of $g$ on homogeneous elements of $A$ is a sum of evaluations of $g$ on homogeneous elements of $B$ or $C$. Those which are only on elements of $C$ are actually zero since $g$ is a graded polynomial identity for $C$. The remaining ones belong to $B$ since $B$ is an ideal of $A$, and then we use that $f$ is an ordinary polynomial identity for $B$ to get the desired conclusion.
\end{proof}

\begin{Corollary}
Let $B$ be a homogeneous ideal of $A$ satisfying a multilinear graded polynomial identity $f(x_{1}^{(g_{1})},\dots,x_{m}^{(g_{m})})$ and let $C$ be a homogeneous subalgebra satisfying some ordinary multilinear polynomial identity $g(x_{1},\dots,x_{n})$. Then $A$ satisfies the following graded identity 
\[
f(g(x_{11}^{(g_{1})},x_{12}^{(1)},\dots,x_{1n}^{(1)}),\dots,g(x_{m1}^{(g_{n})},x_{m2}^{(1)},\dots,x_{mn}^{(1)})).
\]
\end{Corollary}

\begin{proof}
The proof follows the argument from the previous corollary, and hence it will be omitted.  
\end{proof}

\section{\bf Graded semi-identities}

In this section $A=B+C$ will denote an associative algebra over a field $F$ graded by a finite group $G$, and $A$ is a sum of its homogeneous subalgebras $B$ and $C$. Following closely the approach given in \cite{FGL}, we will define the so-called (graded) semi-identities for the sum $A=B+C$ and we will show how some particular semi-identities can ensure the existence of graded polynomial identities for $A$. 

As a motivation for the following definition we observe that $A_g=B_g+C_g$ for each $g\in G$, hence every $a_g\in A_g$ is a sum $a_g=b_g+c_g$, $b_g\in B_g$, $c_g\in C_g$. Thus it is convenient to consider free variables corresponding to the homogeneous components $B_g$ and $C_g$ and look for their interplay with the variables corresponding to $A_g$.

We introduce the following sets of variables. For each $g\in G$ we consider $Y^{(g)}=\{y_{1}^{(g)},y_{2}^{(g)},\dots\}$ and $Z^{(g)}=\{z_{1}^{(g)},z_{2}^{(g)},\dots\}$. Take $Y=\bigcup_{g\in G}Y^{(g)}$ and $Z=\bigcup_{g\in G}Z^{(g)}$ and let $F\langle Y\cup Z|G \rangle$ be the $G$-graded free associative algebra freely generated over $F$ by $Y\cup Z$. We set $x_{i}^{(g)}=y_{i}^{(g)}+z_{i}^{(g)}$ for each $g\in G$, and then we may consider $F\langle X|G \rangle\subset F\langle Y\cup Z|G \rangle$. Next we introduce the definition of graded semi-identities for a sum $A=B+C$.

\begin{Definition}
Let
\[
f=f(y_{1}^{(h_{1})},\dots,y_{m}^{(h_{m})},z_{1}^{(\tilde{h}_{1})},\dots,z_{n}^{(\tilde{h}_{n})})\in F\langle Y\cup Z|G \rangle,
\]
such that $h_{1}$, \dots, $h_{m}$, $\tilde{h}_{1}$, \dots, $\tilde{h}_{n}\in G$.
We say that $f=0$ is a {\it graded semi-identity} for $A=B+C$ if  
\[
f(b_{1}^{(h_{1})},\dots,b_{m}^{(h_{m})},c_{1}^{(\tilde{h}_{1})},\dots,c_{n}^{(\tilde{h}_{n})})=0
\]
for all $b_{i}^{(h_{i})}\in B_{h_{i}}$, $c_{j}^{(\tilde{h}_{j})}\in C_{\tilde{h}_{j}}$. We say that a graded semi-identity $f=0$ is trivial if $f\in Id_{G}(A)$.
\end{Definition}

We observe here that the notion of a (graded) semi-identity depends on the decomposition $A=B+C$.

\begin{Example}\label{example nontrivial semi}
Let $D$ be an $F$-algebra without $1$ such that $D=S_{1}+S_{2}$, where $S_{1}$ is a subalgebra of $D$ and $S_{2}$ is an ideal of $D$ (for instance, take $D=F\langle X \rangle$, the free associative algebra without $1$, $S_1$ the subalgebra generated by all monomials which do not contain the variable $x_{1}$, and $S_{2}$ the subalgebra generated by the monomials that contain $x_{1}$).

We now set 
\[
A=\left(\begin{array}{cc}
D&S_2\\
S_2&D
\end{array}\right).
\]
Consider the $\Z_{2}$-grading on $A$ given by $A=A_{0}\oplus A_{1}$ where 
\[
A_{0}=\left(\begin{array}{cc}
D&0\\
0&D
\end{array}\right) \ \mbox{and} \ A_{1}=\left(\begin{array}{cc}
0&S_2\\
S_2&0
\end{array}\right).
\]
Notice then that $A=B+C$, where
\[
B=\left(\begin{array}{cc}
S_1&0\\
0&S_1
\end{array}\right) \ \mbox{and} \ C=\left(\begin{array}{cc}
S_2&S_2\\
S_2&S_2
\end{array}\right),
\]
and  both $B$ and $C$ are homogeneous subalgebras of $A$. Now, it is clear that $y_{1}^{(1)}=0$ is a non-trivial graded semi-identity for $A$.
\end{Example}

In this section we will present some nontrivial  graded semi-identities satisfied by $A$ that imply the existence of graded polynomial identities for $A$. We now define the type of a multilinear polynomial in $F\langle Y\cup Z|G \rangle$ which we will consider as a semi-identity for $A$.

\begin{Definition}
Let $y_{1}$, \dots, $y_{d}\in Y$ and $x_{d+1}$, \dots, $x_{2d-1}\in X$ be  homogeneous variables. We define the following polynomial
\[
Sp_{d}(y_{1},\dots,y_{d};x_{d+1},\dots,x_{2d-1})=\sum_{\sigma\in S_{d}}\alpha_{\sigma}y_{\sigma(1)}x_{d+1}y_{\sigma(2)}x_{d+2}\cdots x_{2d-1}y_{\sigma(d)}
\]
where $\alpha_{\sigma}\in F$, and $S_d$ stands for the symmetric group permuting $\{1,\ldots, d\}$.
\end{Definition}

The polynomial defined above was influenced by a generalization of the {\it Capelli identity}, the so-called {\it sparse identity} (see for instance \cite{BBRY}). We will be particularly interested in the case where the variables $y$'s are of the same homogeneous degree $g$ and the variables $x$'s are all of homogeneous degree $g^{-1}$. We will denote the latter polynomial as $Sp_{d}[Y^{(g)},X^{(g^{-1})}]$.   

\subsection{Codimensions modulo graded semi-identities}

From now on we assume that $G=\{g_{1},\ldots,g_{k}\}$. 

Given $n\in\mathbb{N}$ we write $n=n_{1}+\cdots+n_{k}$, where $n_{1}$, \dots, $n_{k}$ are non-negative integers. Consider the vector space of multilinear graded polynomials in $n_{i}$ homogeneous variables of degrees $g_{i}$, respectively.
\begin{align*}
P_{n_{1},\dots ,n_{k}}&=span\{u_{\sigma(1)}\cdots u_{\sigma(n)}\mid \sigma\in S_{n}, u_{i_{1}}=x_{i_{1}}^{(g_{1})} \ \mbox{for} \ i_{1}\in\{1,\dots ,n_{1}\},\\ 
&u_{n_{1}+i_{2}}=x_{n_{1}+i_{2}}^{(g_{2})} \ \mbox{for} \ i_{2}\in\{1,\dots,n_{2}\},\dots, \\
&u_{n_{1}+\cdots n_{k-1}+i_{k}}=x_{n_{1}+\cdots n_{k-1}+i_{k}}^{(g_{k})} \ \mbox{for} \ i_{k}\in\{1,\dots,n_{k}\}
 \}.
\end{align*}

Note that $\dim(P_{n_{1},\dots,n_{k}})=n!$ and therefore we have the following straightforward lemma.

\begin{Lemma}\label{l6}
If there exists a positive integer $n=n_{1}+\cdots+n_{k}$ such that 
\[
\dim\bigg(\frac{P_{n_{1},\dots,n_{k}}}{P_{n_{1},\dots,n_{k}}\cap Id_{G}(A)}\bigg)<n!,
\]
then $A$ satisfies a multilinear graded polynomial identity of degree $n$ in $n_{i}$ variables of homogeneous degrees $g_{i}$, respectively.
\end{Lemma}

Denoting by $Id_{G}^{s}(A)$ the set of all graded semi-identities of $A$ (including the identically zero polynomial), it follows immediately that 
$Id_{G}(A)$, $Id_{G}(B)$, $Id_{G}(C)\subset Id_{G}^{s}(A)$. Moreover $Id_{G}^{s}(A)$ is an ideal of $F\langle  Y\cup Z|G \rangle$ which is invariant under all graded endomorphisms that preserve both $F\langle Y|G \rangle$ and $F\langle Z|G \rangle$.

The vector space of multilinear polynomials in $F\langle Y\cup Z|G \rangle$ is defined as

\begin{align*}
V_{n_{1},\dots, n_{k}}&=span\{v_{\sigma(1)}\cdots v_{\sigma(n)}\mid \sigma\in S_{n}, v_{i_{1}}\in\{y_{i_{1}}^{(g_{1})},z_{i_{1}}^{(g_{1})}\} \ \mbox{for} \ i_{1}\in\{1,\dots,n_{1}\},
\ldots,\\ 
& v_{n_{1}+\cdots+ n_{k-1}+i_{k}}\in\{y_{n_{1}+\cdots+n_{k-1}+i_{k}}^{(g_{k})},z_{n_{1}+\cdots+n_{k-1}+i_{k}}^{(g_{k})}\} \ \mbox{for} \ i_{k}\in\{1,\dots,n_{k}\} 
\}.
\end{align*}

Note that $P_{n_{1},\dots,n_{k}}\subset V_{n_{1},\dots,n_{k}}$. Moreover, if  
\[
f\in P_{n_{1},\dots,n_{k}}\cap (V_{n_{1},\dots,n_{k}}\cap Id_{G}^{s}(A)),
\] 
then $f$ is a graded semi-identity in the homogeneous variables $y$ and $z$, and $f$ can be written as a polynomial in the variables  $x_{i}^{(g)}=y_{i}^{(g)}+z_{i}^{(g)}$. Now given any $a_{g}\in A_{g}$, Lemma \ref{l} yields the existence of $b_{g}\in B_{g}$ and $c_{g}\in C_{g}$ such that $a_{g}=b_{g}+c_{g}$. Hence an evaluation of the variable $x_{i}^{(g)}$ on some element $a_{g}$ implies an evaluation of the variables $y_{i}^{(g)}$ on some $b_{g}$ and $z_{i}^{(g)}$ on some $c_{g}$, respectively. Since $f$ is a graded semi-identity, then such evaluation must be zero on $A$. This shows that $f$ is a graded identity. Therefore we conclude that  
\[
P_{n_{1},\dots,n_{k}}\cap Id_{G}(A)=P_{n_{1},\dots,n_{k}}\cap(V_{n_{1},\dots,n_{k}}\cap Id_{G}^{s}(A)).
\]
As a consequence of the above observations we have the following lemma.

\begin{Lemma}\label{l10}
If there exists a positive integer $n=n_{1}+\cdots+n_{k}$ such that 
\[
\dim\bigg(\frac{V_{n_{1},\dots,n_{k}}}{V_{n_{1},\dots,n_{k}}\cap Id_{G}^{s}(A)}\bigg)<n!
\]
then $A$ satisfies some multilinear graded polynomial identity of degree $n$ in $n_{i}$ variables of homogeneous degrees $g_{i}$, for $i=1$, \dots, $k$.
\end{Lemma}

We consider one further decomposition on the space $V_{n_{1},\dots,n_{k}}$. For each $j=1$, \dots, $k$ we consider integers $0\leq r_{j}\leq n_{j}$, and $1\leq t_{j_{1}}<\cdots <t_{j_{r_{j}}}\leq n_{j}$. Denoting ${\bf r}=(r_{1},\dots,r_{k})$ and ${\bf t}=(t_{1_{1}},\dots,t_{1_{r_{1}}},\dots,t_{k_{1}},\dots,t_{k_{r_{k}}})$, we define 
\begin{align*} 
V_{n_{1},\dots,n_{k},\bf{r},\bf{t}}&=span\{v_{\sigma(1)}\cdots v_{\sigma(n)}\mid \sigma\in S_{n}, v_{i_{1}}=y_{i_{1}}^{(g_{1})} \ \mbox{for} \ i_{1}\in\{t_{1_{1}},\dots,t_{1_{r_{1}}}\}, \ \mbox{and} \\
 &v_{i_{1}}=z_{i_{1}}^{(g_{1})} \ \mbox{for} \ i_{1}\in\{1,\dots,n_{1}\}\setminus\{t_{1_{1}},\dots,t_{1_{r_{1}}}\},\dots,\\
 &v_{n_{1}+\cdots+n_{k-1}+i_{k}}=y_{n_{1}+\cdots+n_{k-1}+i_{k}}^{(g_{k})} \ \mbox{for} \ i_{k}\in\{t_{k_{1}},\dots,t_{k_{r_{k}}}\} \ \mbox{and}\\
 &v_{n_{1}+\cdots+n_{k-1}+i_{k}}=z_{n_{1}+\cdots+n_{k-1}+i_{k}}^{(g_{k})} \ \mbox{for} \ i_{k}\in\{1,\dots,n_{k}\}\setminus\{t_{k_{1}},\dots,t_{k_{r_{k}}}\}\}.
\end{align*}

Note that
\[
V_{n_{1},\dots,n_{k}}=\bigoplus_{r_{1}=0}^{n_{1}}\bigoplus_{1\leq t_{1_{1}}<\cdots< t_{1_{r_{1}}}\leq n_{1}}\cdots\bigoplus_{r_{k}=0}^{n_{k}}\bigoplus_{1\leq t_{k_{1}}<\cdots< t_{k_{r_{k}}}\leq n_{k}} V_{n_{1},\dots,n_{k},{\bf r}, {\bf t}}.
\]
In particular, a graded semi-identity can be written as a sum of polynomials in $V_{n_{1},\dots,n_{k},{\bf r}, {\bf t}}$. We recall that in $V_{n_{1},\dots,n_{k},{\bf r}, {\bf t}}$ we have $k$ groups of distinct variables, with $n_1$, \dots, $n_k$ variables in each one of them,  respectively. In the $i$-th group of variables we take $r_i$ among them: the ones with indices $\{t_{i_1}, \ldots, t_{i_{r_i}}\}$, $0\le r_i\le n_i$. These variables are the $y^{(g_i)}$, and the remaining variables from this group are $z^{(g_i)}$.   

\begin{Lemma}\label{l7}
\begin{eqnarray}\nonumber
V_{n_{1},\dots,n_{k}}\cap Id_{G}^{s}(A)
=\bigoplus_{\substack{r_{1}=0\\ 1\leq t_{1_{1}}<\cdots< t_{1_{r_{1}}}\leq n_{1}}}^{n_{1}}\cdots\bigoplus_{\substack{r_{k}=0 \\ 1\leq t_{k_{1}}<\cdots< t_{k_{r_{k}}}\leq n_{k}}}^{n_{k}} (V_{n_{1},\dots,n_{k},{\bf r},{\bf t}}\cap Id_{G}^{s}(A)).
\end{eqnarray}
\end{Lemma}

\begin{proof}
Let $f\in V_{n_{1},\dots,n_{k}}\cap Id_{G}^{s}(A)$ and write
\begin{eqnarray}\label{e1}
f=\sum_{r_{1}=0}^{n_{1}}\sum_{1\leq t_{1_{1}}<\cdots< t_{1_{r_{1}}}\leq n_{1}}\cdots \sum_{r_{k}=0}^{n_{k}}\sum_{1\leq t_{k_{1}}<\cdots<t_{k_{r_{k}}}\leq n_{k}}
f_{{\bf r},{\bf t}}.
\end{eqnarray}
We have to show that each $f_{{\bf r},{\bf t}}$ is also a graded semi-identity of $A$. We proceed by induction on the number of terms of the sum (\ref{e1}). The base case is when  $f=f_{{\bf r},{\bf t}}$, and here we already have $f_{{\bf r},{\bf t}}\in Id_{G}^{s}(A)$. From now on we suppose that there exist at least two nonzero terms in (\ref{e1}), and we take two distinct of them, say $f_{{\bf r},{\bf t}}$ and $\tilde{f}_{{\bf \tilde{r}},{\bf \tilde{t}}}$. Hence there exists some variable, which without loss of generality we will suppose $y_{i}^{(g_{1})}$, such that it occurs in $f_{{\bf r},{\bf t}}$ but not in $\tilde{f}_{{\bf \tilde{r}},{\bf \tilde{t}}}$. Now we write $f=f_{1}+f_{2}$, where $f_{1}$ is the sum of the terms from (\ref{e1}) that  contain the variable $y_{i}^{(g_{1})}$ and $f_{2}$ is the sum of terms from (\ref{e1}) which do not. Evaluating $y_{i}^{(g_{1})}$ by $0$ we obtain that $f_{2}$ is a consequence of $f$ and therefore $f_{2}$ is a graded semi-identity. Hence the same happens to $f_{1}=f-f_{2}$. Now it is enough to apply the induction hypothesis to both $f_{1}$ and $f_{2}$.
\end{proof}

We finish this section by showing how the decomposition above can be used to prove the existence of a graded polynomial identity for $A$.

In order to simplify our notation we will write just $V_{n_{1},\dots,n_{k},{\bf r}}$ instead of  $V_{n_{1},\dots,n_{k},{\bf r},{\bf t}}$, where ${\bf t}=(1,\dots,r_{1},\dots,1,\dots,r_{k})$. Note that  given any ${\bf t}$ there exists a graded isomorphism of vector spaces  $V_{n_{1},\dots,n_{k},{\bf r}}\cong V_{n_{1},\dots,n_{k},{\bf r},{\bf t}}$ such that 
\[
V_{n_{1},\dots,n_{k},{\bf r}}\cap Id_{G}^{s}(A)\cong V_{n_{1},\dots,n_{k},{\bf r},{\bf t}}\cap Id_{G}^{s}(A).
\]

We also note that there exist exactly $\displaystyle \binom{n_{1}}{r_{1}}\cdots\binom{n_{k}}{r_{k}}$ vector spaces $V_{n_{1}, \dots, n_{k},{\bf r},{\bf t}}$ isomorphic to $V_{n_{1},\dots,n_{k},{\bf r}}$.  

\begin{Lemma}\label{l8}
If there exists a positive integer $n=n_{1}+\cdots+n_{k}$ such that
\begin{eqnarray}\nonumber
\dim\bigg(\frac{V_{n_{1},\dots,n_{k},{\bf r}}}{V_{n_{1},\dots,n_{k},{\bf r}}\cap  Id_{G}^{s}(A)}\bigg)<\frac{n!}{2^{n}},
\end{eqnarray}
for all ${\bf r}$, then $A$ satisfies some multilinear graded polynomial identity of degree $n$ in $n_{i}$ variables of homogeneous degrees $g_{i}$, for $i=1$, \dots, $k$.
\end{Lemma}

\begin{proof}
By Lemma \ref{l7} we have 
\begin{align*}
\dim\frac{V_{n_{1}, \dots,n_{k}}}{V_{n_{1},\dots,n_{k}}\cap Id_{G}^{s} (A)}
&=\sum_{r_{1}=0}^{n_{1}}\cdots \sum_{r_{k}=0}^{n_{k}}\binom{n_{1}}{r_{1}}\cdots \binom{n_{k}}{r_{k}} \dim \frac{V_{n_{1},\dots,n_{k},{\bf r}}}{V_{n_{1},\dots,n_{k},{\bf r}}\cap Id_{G}^{s}(A)} \\
&< \sum_{r_{1}=0}^{n_{1}}\cdots \sum_{r_{k}=0}^{n_{k}}\binom{n_{1}}{r_{1}}\cdots \binom{n_{k}}{r_{k}}\frac{n!}{2^n}\\
&=2^{n_{1}}\cdots 2^{n_{k}}\frac{n!}{2^n}=n!
\end{align*}
Now it suffices to apply Lemma \ref{l10}.
\end{proof}

\subsection{A generating set for $V_{n,0,\dots,0,\bf{r}}$ modulo $Id_{G}^{s}(A)$}

In this section we give a generating set for $V_{n,0,\dots,0,\bf{r}}$ by using the language of good permutations of the symmetric group. This notion appeared in the proof of the well known theorem of Regev about the exponential upper bound of the codimension sequence of a PI algebra, see \cite{regev}. Good sequences and similar combinatorial notions have been extensively used in studying numerical invariants of PI algebras such as codimensions, cocharacters, and so on.

\begin{Definition}
Let $n\in\mathbb{N}$ and $1\leq d\leq n$. We say that $\sigma\in S_{n}$ is a $d$-bad permutation if there exist $1\leq i_{1}<\cdots< i_{d}\leq n$ such that   $\sigma(i_{1})>\cdots>\sigma(i_{d})$. Otherwise we call $\sigma \in S_{n}$ a $d$-good permutation.
\end{Definition}

We recall a well known result about the number of $d$-good permutations (see \cite{Reg}, for instance). It can be obtained by using the well known theorem of Dilworth in Combinatorics.

\begin{Lemma}
The number of $d$-good permutations in $S_{n}$ is at most $(d-1)^{2n}$.
\end{Lemma}

We will adopt the following convention: given $d>n$, then every permutation in $S_{n}$ is $d$-good. Such convention does not change the maximum number of $d$-good permutations. Indeed, in this case the number is exactly $n!$ which is less than $(d-1)^{2n}$.

\begin{Definition}
Let $m$ be the monomial 
\[
v_{1}y_{\sigma(1)}^{(h_{1})}\cdots y_{\sigma(i_{1})}^{(h_{i_{1}})}v_{2}y_{\sigma(i_{1}+1)}^{(h_{i_{1}+1})}\cdots y_{\sigma(i_{2})}^{(h_{i_{2}})}v_{3}\cdots v_{l}y_{\sigma(i_{l-1}+1)}^{(h_{i_{l-1}+1})}\cdots y_{\sigma(i_{l})}^{(h_{i_{l}})}v_{l+1} \in V_{n_{1},\dots,n_{k},{\bf r},{\bf t}}
\]
where  $v_{1}$, \dots, $v_{l+1}$ are (eventually empty) words in the homogeneous variables of type $z$. We say that $m$ is a $d$-$y$-good monomial if the permutation $\sigma\in S_{l}$ is a $d$-good one. 
\end{Definition}

We recall the following combinatorial fact concerning finite groups,  \cite{BGR}, Lemma 4.1.

\begin{Lemma}\label{l1}
Every product of $|H|d$ words in a finite group $H$ contains a product of $d$ consecutive trivial subwords.
\end{Lemma}

In the next lemma we will assume that $A$ satisfies a graded semi-identity of the form  $Sp_{d_{1}}[Y^{(g)};X^{(g^{-1})}]$, for some $g\in G$. Without loss of generality we write $g=g_{1}$ and we denote $n_{1}=n$ and $r_{1}=r$.  We also use $o(g)$ to denote the order of the element $g\in G$ (and of the cyclic subgroup generated by $g$).

\begin{Lemma}\label{l9}
The space $V_{n,0,\dots,0,{\bf r}}$ is generated, modulo $Id_{G}^{s}(A)$, by all  $(2d_{1}-1)o(g)$-$y$-good monomials for all ${\bf r}$.
\end{Lemma}

\begin{proof}
Following our convention we may assume $(2d_{1}-1)o(g)\leq r$.

The proof will be done by contradiction. Thus we assume that $V_{n,0,\dots,0,{\bf r}}$ is not generated by the $(2d_{1}-1)o(g)$-$y$-good monomials modulo $Id_{G}^{s}(A)$. Hence the set $\mathfrak{U}$ of all $(2d_{1}-1)o(g)$-$y$-bad monomials which cannot be written as a linear combination of $(2d_{1}-1) o(g)$-$y$-good monomials modulo $Id_{G}^{s}(A)$ is nonempty. We order the variables in $Y^{(g)}$ as $y_{1}^{(g)}<y_{2}^{(g)}<\cdots$ and we take in  $\mathfrak{U}$ the partial order given lexicographically from the left to the right in the variables in $Y^{(g)}$ only. Let $m_{\tau}\in\mathfrak{U}$ be a minimal element, where $\tau\in S_{r}$ is the permutation which defines the positions of the variables from $Y^{(g)}$ in this minimal element. In particular, $m_{\tau}$ is a $(2d_{1}-1)o(g)$-$y$-bad monomial and hence there exist $1\leq i_{1}<\cdots< i_{(2d_{1}-1)o(g)}\leq r$ such that  $\tau(i_{1})>\cdots>\tau(i_{(2d_{1}-1)o(g)})$. 

We write $m_{\tau}=w_{0}w_{1}w_{2}\cdots w_{(2d_{1}-1)o(g)}w_{(2d_{1}-1)o(g)+1}$, where $w_{j}$ is the word that starts with $y_{\tau(i_{j})}^{(g)}$ and ends just before $y_{\tau(i_{j+1})}^{(g)}$, $j=1$, \dots, $(2d_{1}-1)o(g)-1$, $w_{(2d_{1}-1)o(g)}=y_{\tau(i_{(2d_{1}-1)o(g)})}^{(g)}$, $w_{0} $ and $w_{(2d_{1}-1)o(g)+1}$ are suitable words.

Applying Lemma \ref{l1} to the subgroup generated by $g$, we obtain the existence of $2d_{1}-1$ consecutive words  $\bar{w}_{1}$, \dots, $\bar{w}_{2d_{1}-1}$ among  $w_{1}$, \dots, $w_{(2d_{1}-1)o(g)}$ where $\deg(\bar{w}_{j})=1$, $j=1$, \dots, $2d_{1}-1$.

Hence we rewrite $m_{\tau}=\bar{w}_{0}\bar{w}_{1}\cdots \bar{w}_{2d_{1}-1}\bar{w}_{2d_{1}}$, where $\bar{w}_{0}$ and $\bar{w}_{2d_{1}}$ are suitable words. Note that for each $j=1$, \dots, $d_{1}-1$, we can write $\bar{w}_{2j-1}\bar{w}_{2j}=y_{\tau(i_{l})}^{(g)}m_{d_{1}+j}$ for some $l$ (which depends on $j$) and $m_{d_{1}+j}$ is the subword of $\bar{w}_{2j-1}\bar{w}_{2j}$ obtained by deleting its first variable.  Denote $m_{j}=y_{\tau(i_{l})}^{(g)}$ for $j=1$, \dots, $d_{1}-1$, and take $m_{d_{1}}$ as the first variable of $\bar{w}_{2d_{1}-1}$.

In this way we have 
\[
m_{\tau}=m_{0}m_{1}m_{d_{1}+1}m_{2}\cdots m_{d_{1}-1}m_{2d_{1}-1}m_{d_{1}}m_{2d_{1}}
\]
where $m_{0}=\bar{w}_{0}$, $m_{2d_{1}}$ is a suitable word, $\deg(m_{j})=g$ for $j=1$, \dots, $d_{1}$ and $\deg(m_{d_{1}+j})=g^{-1}$ for $j=1$, \dots, $d_{1}-1$. Indeed, it follows from $\deg(\bar{w}_{2j-1}\bar{w}_{2j})=1$ and $\deg(y_{\tau(i_{l})}^{(g)})=g$ that $\deg(m_{d_{1}+j})=g^{-1}$. We also note that each $m_{j}$ is evaluated on $Y^{(g)}$ and each $m_{d_{1}+j}$ is evaluated on $X^{(g^{-1})}$.  Moreover, we have $m_{1}>\cdots>m_{d_{1}}$.

Recalling that $Sp_{d_{1}}[Y^{(g)};X^{(g^{-1})}]\in Id_{G}^{s}(A)$ we can write  
\[
m_{\tau}=\sum_{\sigma\in S_{d_{1}}\setminus\{id\}}-\alpha_{\sigma}m_{0}m_{\sigma(1)}m_{d_{1}+1}m_{\sigma(2)}\cdots m_{2d_{1}-1}m_{\sigma(d_{1})}m_{2d_{1}} \pmod{Id_{G}^{s}(A)}.
\]
Since for each $\sigma\in S_{d_{1}}\setminus\{id\}$ we have 
\[
m_{\sigma}=m_{0}m_{\sigma(1)}m_{d_{1}+1}m_{\sigma(2)}\cdots m_{2d_{1}-1}m_{\sigma(d_{1})}m_{2d_{1}}<m_{\tau},
\]
the minimality of $m_{\tau}$ leads us to  $m_{\sigma}\notin\mathfrak{U}$. Therefore each $m_{\sigma}$ can be written as a linear combination of $(2d_{1}-1)o(g)$-$y$-good monomials modulo the graded semi-identities of $A$, and hence the same happens for $m_{\tau}$. This is a contradiction to $m_{\tau}\in\mathfrak{U}$. 
\end{proof}

\subsection{Existence of a graded identity for $\mathcal{A}$}
First of all let us estimate the dimension of $V_{n,0,\dots,0,\bf{r}}$ modulo $Id_{G}^{s}(A)$. We start by recalling the following result from \cite{BGR}.

\begin{Theorem}\label{Riley}
Let $A$ be an algebra graded by a finite group $G$. If the neutral component of $A$ satisfies a polynomial identity of degree $d$ then
\[
\dim\frac{P^{(h_{1},\dots,h_{n})}}{P^{(h_{1},\dots,h_{n})}\cap Id_{G}(A)} \leq (|G|d-1)^{2n}
\]
for every $n\in\mathbb{N}$ and $(h_{1},\dots,h_{n})\in G^{n}$.
\end{Theorem}

For our main goal of this section we consider one last decomposition of the vector space $V_{n,0,\dots,0,{\bf r}}$ into the subspaces
\begin{align*}
U_{n,0,\dots,0,{\bf r},u,{\bf p},{\bf q}}&=span\{y_{\sigma(1)}\dots y_{\sigma(p_{1})}z_{\tau(1)}\cdots z_{\tau(q_{1})} 
\cdots y_{\sigma(p_{1}+\cdots+ p_{u-1}+1)}\cdots y_{\sigma(r)}\times\\
&\times z_{\tau(q_{1}+\cdots+ q_{u-1}+1)}\cdots z_{\tau(n-r)}\mid 
\sigma\in S_{r},\tau\in S_{n-r}\}
\end{align*}
where $r=p_{1}+\cdots+p_{u}$ and $n-r=q_{1}+\cdots+q_{u}$ are such that $p_{1}\geq0$, $q_{u}\geq0$, and $ p_{2}$, \dots, $p_{u}$, $q_{1}$, \dots, $q_{u-1}>0$, and the homogeneous variables were written without their homogeneous degrees for simplicity.

Hence we have
\begin{equation}
\label{e2}
V_{n,0,\dots,0,{\bf r}}=\bigoplus_{u,{\bf p},{\bf q}} U_{n,0,\dots,0,{\bf r},u,{\bf p},{\bf q}}.
\end{equation}

Moreover, since the number of compositions of a given positive integer $n$ is exactly $2^{n-1}$, one can see that there exist at most $2^{r}\cdot 2^{n-r}=2^{n}$ vector spaces of the same type as  $U_{n,0,\dots,0,{\bf r},u,{\bf p},{\bf q}}$.

\begin{Lemma}\label{10}
\[
\dim\frac{V_{n,0,\dots,0,{\bf r}}} {V_{n,0,\dots,0,{\bf r}}\cap Id_{G}^{s}(A)}  
\leq 2^{n}((2d_{1}-1)o(g)-1)^{2r} (|G|d_{2}-1)^{2(n-r)}(r+1)^{n-r}
\]
\end{Lemma}

\begin{proof}
We start the proof by recalling that  Lemma \ref{l9} gives us a generating set for $V_{n,0,\dots,0,{\bf r}}$ modulo $Id_{G}^{s}(A)$, namely the set of all $(2d_{1}-1)o(g)$-$y^{(g)}$-good monomials. By (\ref{e2}), it is enough to count the number of $(2d_{1}-1)o(g)$-$y^{(g)}$-good monomials in $U_{n,0,\dots,0,{\bf r},u,{\bf p},{\bf q}}$. First of all note that there exist at most $((2d_{1}-1)o(g)-1)^{2r}$ dispositions of the homogeneous variables $y$ modulo $Id_{G}^{s}(A)$. We also have  $\binom{n-r}{q_{1},\dots,q_{u}}$ different manners of distributing the homogeneous variables $z$ that occur in blocks of $q_{1}$ consecutive ones, \dots, $q_{u}$ consecutive ones. By Theorem \ref{Riley}, each consecutive group of homogeneous variables $z$ can be written as a linear combination of at most $(|G|d_{2}-1)^{2q_{i}}$ monomials modulo $Id_{G}^{s}(A)$. Hence  $U_{n,0,\dots,0,{\bf r},u,{\bf p},{\bf q}}$ is generated by at most
\[ 
((2d_{1}-1)o(g)-1)^{2r}\binom{n-r}{q_{1},\dots,q_{u}}(|G|d_{2}-1)^{2q_{1}}\cdots (|G|d_{2}-1)^{2q_{u}}
\]
monomials.

Therefore the well known multinomial theorem enables us to bound the multinomial coefficient corresponding to the homogeneous variables $z$ by $u^{n-r}$ and since $u\leq r+1$ we write 
\[
\binom{n-r}{q_{1},\dots,q_{u}}\leq (r+1)^{n-r}.
\]

We finish the proof of this lemma by observing that there exist at most  $2^{n}$ vector spaces of the form  $U_{n,0,\dots,0,{\bf r},u,{\bf p},{\bf q}}$.
\end{proof}

We are ready to prove the main result of this section. We recall the following well known inequality (see \cite{FRo})

\begin{eqnarray}
\bigg(\frac{n}{e}\bigg)^{n}<\frac{\Gamma(n+1)}{\sqrt{2\pi n}}<\Gamma(n+1)=n!,
\end{eqnarray}
where $\Gamma$ stands for the Euler's gamma function, and the meaning of $e$ and $\pi$ is the obvious one.

\begin{Theorem}\label{t2}
Let $A=B+C$ be a $G$-graded algebra which is a sum of two homogeneous subalgebras. If $B$ satisfies the graded semi-identity $Sp_{d_{1}}[Y^{(g)};X^{(g^{-1})}]$ and $C_{1}$ satisfies some ordinary polynomial identity of degree $d_{2}$, then $A$ satisfies some graded multilinear polynomial identity of degree $n$ in homogeneous variables of degree $g$, where $n$ is the least integer greater or equal to $\alpha^\alpha$ such that $\alpha=8e((2d_{1}-1)o(g)-1)^{2}(|G|d_{2}-1)^{2}$. 
\end{Theorem}

\begin{proof}
Lemma \ref{l8} reduces the existence of the required graded polynomial identity to the existence of a positive integer $n\in\mathbb{N}$ such that 
\begin{eqnarray}\nonumber 
\dim\frac{V_{n,0,\dots,0,{\bf r}}}{V_{n,0,\dots,0,{\bf r}}\cap Id_{G}^{s}(A)} <\frac{n!}{2^{n}},
\end{eqnarray}
for every ${\bf r}$.

We begin by assuming that $r\neq0$ and we will show that there exists $n\in\mathbb{N}$ (which does not depend on $r$) such that 
\begin{eqnarray}\nonumber
8^{n}e^{n}((2d_{1}-1)o(g)-1)^{2n}(|G|d_{2}-1)^{2n}r^{n-r}\leq n^{n}
\end{eqnarray}

Recall $\alpha=8e((2d_{1}-1)o(g)-1)^{2}(|G|d_{2}-1)^{2}$. Take $n\in\mathbb{N}$ as the least integer satisfying  $n\geq\alpha^\alpha$, we claim that  $\alpha^{n}r^{n-r}\leq n^{n}$. We can consider $r<n$, since for $r=n$ the inequality $\alpha^{n}\leq n^{n}$ follows from $\alpha<n$. We consider two cases now.

\begin{enumerate}
\item[\textbf{Case 1:}] $\displaystyle r\leq \frac{n}{\alpha}$.

In this case we have $\alpha r\leq n$ and hence $\alpha^{n}r^{n-r}\leq (\alpha r)^{n}\leq n^{n}$. 

\item[\textbf{Case 2:}] $\displaystyle \frac{n}{\alpha}<r< n$.

In this second case we have $n<\alpha r$, that is, there exists a positive real number $v$ such that $\alpha r=n+v$. The minimality of $n$ leads us to $r<\alpha^{\alpha}$. We also note that $\alpha^{v}>1$. Therefore,

\begin{align*}
n^{n}&\geq (\alpha^{\alpha})^{n}=\alpha^{n}(\alpha^{\alpha-1})^{n-r}(\alpha^{\alpha-1})^{r}=\alpha^{n}(\alpha^{\alpha-1})^{n-r}\alpha^{\alpha r -r}\\
&=\alpha^{n}(\alpha^{\alpha-1})^{n-r}\alpha^{n-r+v}
=\alpha^{n}(\alpha^{\alpha})^{n-r}\alpha^{v}\\
&>\alpha^{n}r^{n-r}
\end{align*}
\end{enumerate}

Since $r\neq0$, we have $r+1\leq 2r$ and hence
\begin{align*}
&4^{n}e^{n}((2d_{1}-1)o(g)-1)^{2r}(|G|d_{2}-1)^{2(n-r)}(r+1)^{n-r}\\
\leq & 8^{n}e^{n}((2d_{1}-1)o(g)-1)^{2n}(|G|d_{2}-1)^{2n}r^{n-r}\\
\leq & n^n
\end{align*}
which implies in turn that
\begin{align*}
2^{n}((2d_{1}-1)o(g)-1)^{2r}(|G|d_{2}-1)^{2(n-r)}(r+1)^{n-r}\leq \frac{n^n}{e^n 2^n}<\frac{n!}{2^n}.
\end{align*}

The only case left is when $r=0$; we note that for any positive integer $n\geq 4e(|G|d_{2}-1)^{2}$ (in particular for the same $n$ chosen in the case $r\neq0$) we have $4^n e^n (|G|d_{2}-1)^{2n}\leq n^{n}$,
and then it follows that
\[
2^n (|G|d_{2}-1)^{2n}\leq \frac{n^n}{e^n 2^n}<\frac{n!}{2^n}
\]
Now it is enough to apply Lemma \ref{10} in order to finish the proof.
\end{proof}

\begin{Remark}
If we take $D$, $S_{1}$, $S_2$ as defined in Example \ref{example nontrivial semi} for $D$ given by the free associative algebra $F\langle X \rangle$ without $1$, then $A$ does not satisfy an identity in variables of homogeneous degree $1$, but satisfies a graded semi-identity of the type $Sp_{d}$. However, this does not contradict our last theorem, since clearly $C$ is not PI. 
\end{Remark}

\section{\bf A final remark}

We finish this paper with a final remark about the involutive version of Question \ref{question}. In what follows we recall some basic facts about (identities with) involutions on rings.   

\begin{Definition}
An {\it involution} on an associative ring $R$ is an antiautomorphism $*: R\rightarrow R$ of order two.
\end{Definition}

We consider two sets of variables $X=\{x_{1},x_{2},\dots,\}$ and $X^{*}=\{x_{1}^{*},x_{2}^{*},\dots\}$, then we form the free ring $\mathbb{Z}\langle X, X^{*} \rangle$. An involution (which will be denoted also by $*$) on $\mathbb{Z}\langle X,X^{*} \rangle$ can be defined setting $(x_{i})^{*}=x_{i}^{*}$ and $(x_{i}^{*})^{*}=x_{i}$. Hence we can easily see that a polynomial  $f\in\mathbb{Z}\langle X,X^{*} \rangle$ can be written as 

\begin{eqnarray}\label{invol}
f(x_{1},\dots,x_{n};x_{1}^{*},\dots,x_{n}^{*})=\sum_{i,j}\alpha_{i,j}x_{i_{1}}^{\epsilon_{i_{1}}}\cdots x_{i_{j}}^{\epsilon_{i_{j}}}, \quad \alpha_{i,j}\in\mathbb{Z}, \quad \epsilon_{i_{k}}\in\{1,*\}.
\end{eqnarray}

\begin{Definition}
Let $f\in\mathbb{Z}\langle X,X^{*}\rangle$ be a monic polynomial. We say that $f=0$ is a {\it $*$-polynomial identity} for $(R,*)$ if $f(r_{1},\dots,r_{n},r_{1}^{*},\dots,r_{n}^{*})=0$ for every $r_{1}$, \dots, $r_{n}\in R$. In this case $R$ is called a {\it $*$-PI ring}.
\end{Definition}

Taking $\epsilon_{i_{k}}=1$ in (\ref{invol}) we have that every PI ring with an involution $*$ is also a $*$-PI ring. Thanks to Amitsur's theorem (see \cite{Ami}), the converse also holds.

\begin{Theorem}
If $(R,*)$ satisfies some $*$-polynomial identity, then $R$ is a PI ring.
\end{Theorem}

Now we combine Amitsur's and Kepczyk's theorems to get the following result.

\begin{Theorem}
Let $(R,*)$ be a ring with an involution. Assume that $(R,*)=(R_{1},*_{1})+(R_{2},*_{2})$ is a sum of two subrings with involutions $*_{1}$ and $*_{2}$, respectively. If $R_{i}$ is an $*_{i}$-PI ring for $i=1$, 2, then $R$ is a $*$-PI ring.
\end{Theorem}


\begin{thebibliography}{100}

\bibitem{Ami} Amitsur, S. {\it Identities in rings with involution}, Israel J. Math. {\bf 7}, 63--68 (1969).

\bibitem{BGi} Bahturin, Yu., Giambruno, A. {\it Identities of sums of commutative subalgebras}. Rend. Mat. Palermo {\bf 43}, 250--258 (1994).

\bibitem{BGR} Bahturin, Yu., Giambruno, A., Riley, D. {\it Group-graded algebras with polynomial identity}, Israel J. Math. {\bf 104}, 145--155 (1998).

\bibitem{BMi} Beidar, K., Mikhalev, A. {\it Generalized polynomial identities and rings which are sums of two subrings}, Algebra Logic {\bf 34}, 3--11 (1995).

\bibitem{BBRY} Belov A., Bokut L., Rowen L., Yu J.-T. {\it The Jacobian Conjecture, Together with Specht and Burnside-Type Problems}. In: Cheltsov I., Ciliberto C., Flenner H., McKernan J., Prokhorov Y., Zaidenberg M. (eds) Automorphisms in Birational and Affine Geometry. Springer Proceedings in Mathematics \& Statistics, {\bf 79}. Springer, Cham. (2014).

\bibitem{BCo} Bergen, J., Cohen, M. {\it Actions of commutative Hopf algebras}, Bull. London Math. Soc. {\bf 18}, 159--164 (1986).

\bibitem{eldkoc} Elduque,  A.,  Kochetov, M. \emph{Gradings on simple Lie algebras}, Mathematical Surveys and Monographs, 189. American Mathematical Society (2013).

\bibitem{FRo} Farrell, O., Ross, B. {\it Solved Problems: Gamma and Beta Functions, Legendre Polynomials, Bessel Functions}. Dover Publications, NY, 1971.

\bibitem{FGL} Felzenszwalb, B., Giambruno, A., Leal, G. {\it On rings which are sums of two PI-subrings: a combinatorial approach}, Pacific J. Math. {\bf 209}(1), 17--30 (2003).

\bibitem{Keg} Kegel, O. {\it Zur Nilpotenz gewisser assoziativer Ringe}, Math. Ann. {\bf 149}, 258--260 (1963).

\bibitem{kemer}
Kemer, A. \textit{Ideals of Identities of Associative Algebras}, Translations Math. Monographs, vol. \textbf{87}, Amer. Math. Soc., Providence, RI, 1991.

\bibitem{Kep} K\c{e}pczyk, M. {\it A ring which is a sum of two PI subrings is always a PI ring}, Israel J. Math. {\bf 221}(1), 481--487 (2016).

\bibitem{KPu} K\c{e}pczyk, M., Puczylowski, E. {\it Rings which are sums of two subrings satisfying polynomial identities}, Commun. Algebra {\bf 29}, 2059--2056 (2001).

\bibitem{regev} 
Regev, A. \textit{Existence of identities in $A\otimes B$}, Israel J. Math. \textbf{11}, 131--152 (1972).

\bibitem{Reg} Regev, A.  {\it The representations of $S_{n}$ and explicit identities for P.I. algebras}, J. Algebra {\bf 51}, 25--40 (1978).
\end{thebibliography}

\end{document}